\documentclass[reqno]{amsart}
\usepackage{amsaddr}
\usepackage{amsfonts, amsmath, amssymb, amsthm}

\usepackage{mathptmx}

\usepackage[top=3cm, bottom=3cm, left=3cm, right=3cm]{geometry}

\usepackage{multirow,multicol,diagbox}
\usepackage[table]{xcolor}

\usepackage[square,sort,comma,numbers]{natbib}

\newtheorem{theorem}{Theorem}
\newtheorem{cor}{Corollary}
\newtheorem{lemma}{Lemma}

\theoremstyle{definition}

\newtheorem*{remark}{Remark}
\theoremstyle{remark}

\usepackage{enumitem}

\usepackage[T1]{fontenc}
\DeclareMathAlphabet{\mathpzc}{OT1}{pzc}{m}{it}
\usepackage{yfonts}
\usepackage[sans]{dsfont}
\usepackage{txfonts}
\usepackage[mathscr]{euscript}
\usepackage{bbm}

\newcommand{\nat}{\mathbb N}			%% natural numbers

\newcommand{\ff}[1]{{\mathbb F}_{#1}}		%% ff{q} = GF(q)
\newcommand{\ffs}[1]{{\mathbb F}_{#1}^\star}	%% nonzero elements
\newcommand{\ffx}[1]{\ff{#1}[x]}		%% F[x]
\newcommand{\gen}{\mathfrak{z}}			%% generator, change to whatever

\newcommand{\cyc}[1]{\langle #1 \rangle}	%% cyclic group notation
\date{}

\begin{document}

\title{Bounds on the differential uniformity of the Wan-Lidl polynomials}
\author{Li-An Chen and Robert S. Coulter}
\address{Department of Mathematical Sciences\\
University of Delaware\\
Newark DE 19716, USA}

\begin{abstract}
We study the differential uniformity of the Wan-Lidl polynomials over finite
fields.
A general upper bound, independent of the order of the field, is established.
Additional bounds are established in settings where one of the parameters is
restricted. In particular, we establish a class of permutation polynomials
which have differential uniformity at most 5 over fields of order $3\bmod 4$,
irrespective of the field size.
Computational results are also given.
\end{abstract}

\maketitle

\section{Introduction and The Main Results}

Throughout this paper $\ff{q}$ denotes the finite field of order $q$, with
$q=p^e$ for some prime $p$ and $e\in\nat$, and $\ffs{q}$ denotes the nonzero elements of $\ff{q}$.
We use $\gen$ to denote a primitive element of $\ff{q}$. 
It follows from Lagrange Interpolation and counting that any function on
$\ff{q}$ can be represented uniquely by a polynomial in $\ffx{q}$ of degree
less than $q$.
A polynomial $f\in\ffx{q}$ is called a \textit{permutation polynomial (PP)} over
$\ff{q}$ if the evaluation map $c\mapsto f(c)$ is a bijection on $\ff{q}$.
An easy to prove class of examples come from the monomials:
$x^d$ is a PP over $\ff{q}$ if and only if $\gcd(d,q-1)=1$.
Permutation polynomials have been studied extensively for decades.
A broad introduction is given by Lidl and Niederreiter
\cite{blidl83}, Chapter 7, and there
are the two classical survey articles of Lidl and Mullen \cite{lidl88,lidl93}.
For a more recent survey, see Hou \cite{hou15b}. 

This paper is concerned with an important property of functions known as
differential uniformity.
Let $f\in\ffx{q}$ and $a\in\ffs{q}$. The
{\em differential operator of $f$ in the direction of $a$} is the
function $\Delta_{f,a}(x):=f(x+a)-f(x)$.
The \textit{differential uniformity (DU)} of $f$ is defined by 
\[\delta_{f}=\max_{a\in\ffs{q},c\in\ff{q}}|\{x\in\ff{q}\mid \Delta_{f,a}(x)=c\}|.\]
The lower the DU of a function, the more resistant the function is to
differential attacks when used as an S-box.
Functions with optimal DU are called \textit{almost perfect nonlinear (APN)} over fields of characteristic $2$ (with $2$-DU), and \textit{planar} over fields of odd
characteristic (with $1$-DU). 

PPs with optimal differential uniformity are highly desirable.
Over fields of order $2^e$, there are a number of classes of APN PPs known when
$e$ is odd, but when $e$ is even we have only a single example, found
by Browning, Dillon, McQuistan and Wolfe \cite{bro10} for $e=6$.
If one weakens the requirement to constructing PPs with near-optimal DU, then further examples are known, a recent example being the 6-DU permutations
constructed by Calderini \cite{cal21}.

In odd characteristic, it is actually impossible for a planar function to be
a PP. Indeed, Coulter and Senger \cite{coulter14c} showed that the image set
of a planar function over $\ff{q}$ can be no larger than roughly
$q-\sqrt{q}$. Thus, in odd characteristic, the problem becomes that of finding
PPs with near-optimal DU. Some monomial examples were established in 1997.
Helleseth and Sandberg \cite{helleseth97} showed that the
monomials $x^d$ with $d=(p^e+3)/2$, $p\ne 3$ and $p^e\equiv 3\bmod 4$, are
4-DU over $\ff{p^e}$, see \cite{helleseth97}, Theorem 3. These monomials are
always PPs over $\ff{p^e}$ under these conditions.

Here we focus on a class of PPs first classified by Wan and Lidl in 1991.
A {\em Wan-Lidl polynomial} is any polynomial of the
form $x^s h(x^{(q-1)/d})$ with $h\in\ffx{q}$, $s,d\in\nat$, and $d\mid (q-1)$.
Wan and Lidl gave necessary and sufficient conditions for polynomials of this
form to be PPs in \cite{wan91}, see Theorem \ref{th:WanLidl} below.
Here we give several results concerning the differential uniformity of these
polynomials.
Our main result gives a general upper bound on the DU of a Wan-Lidl polynomial.
\begin{theorem}\label{th:DU_general_UB}
Let $s,d\in\nat$ with $s>1$ and $d\mid (q-1)$.
Let $h\in\ffx{q}$, $T(x)=x^{(q-1)/d}$, and set $f(x)=x^s h(T(x))$. Then $\delta_f\le d(sd-1)+2$. 
\end{theorem}
Note that this bound does not require $f$ to be a PP.
The proof is based on a worst-case scenario which we believe rarely occurs, so
the bound is almost certainly not tight in many cases.
Specialising, we fix the parameter $d=2$ and prove the following result,
establishing an infinite class of binomial permutations with DU at most $5$.
\begin{theorem}\label{th:DU_sEven_d2}
Let $q\equiv 3\pmod 4$, $s\in\nat$ be even. Let $h(x)=x+b$ for some $b\in\ffs{q}$ and $T(x)=x^{(q-1)/2}$. If $f(x)=x^s h(T(x))$ is a PP over $\ff{q}$, then $\delta_f\le 4s-3$. 
\end{theorem}
In particular, when $s=2$ and $s=4$, the bound in Theorem \ref{th:DU_sEven_d2}
gives $\delta_f\le 5$ and $\delta_f\le 13$, respectively.
The bound has been shown to be tight for $s=2$ and $s=4$ using the
Magma algebra system \cite{magma}, though it would appear that the two cases
are very different. When $s=2$, it seems the bound is always tight for
fields of order larger than 59, while for $s=4$, we have found only one example
where the bound is met, over the field of order 3671.
Additionally, computational evidence led us to proving the following corollary,
which gives an infinite class of PPs having DU at most 4.
\begin{cor}\label{co:DU_s2d2b3}
Let $q\equiv 3\pmod 8$. Then $f(x)=x^2 (x^{\frac{q-1}{2}}\pm 3)$ is a PP over $\ff{q}$, and $\delta_f\le 4$. 
\end{cor}
Though the evidence is not particularly strong, it is possible that the PPs of this corollary form the only infinite class of Wan-Lidl PPs with a differential uniformity of $4$. In subsequent computing for $s\in\{4,6\}$ we stopped finding Wan-Lidl PPs with a DU of 4 when the field size got large enough. 

Similarly, by fixing the parameter $s=2$, we obtain the following result.
\begin{theorem}\label{th:DU_s2_dEven}
Let $q$ be odd, $d\in\nat$ be even and $(q-1)/d$ be odd. Let $h\in\ffx{q}$
and $T(x)=x^{(q-1)/d}$. If $f(x)=x^2 h(T(x))$ is a PP over $\ff{q}$, then $\delta_f\le 2d^2-\frac{3}{2}d$. 
\end{theorem}
The paper is organized as follows.
In Section \ref{sc:WLPPs}, we recall the PP classification of Wan-Lidl
polynomials obtained by Wan and Lidl, and explain why we believed these
polynomials warranted further investigation with regard to their DU.
In Section \ref{sc:general}, we prove our general result, Theorem
\ref{th:DU_general_UB}.
In Section \ref{sc:special}, we prove DU bounds for some special cases of
Wan-Lidl PPs, namely, Theorem \ref{th:DU_sEven_d2}, Corollary
\ref{co:DU_s2d2b3}, and Theorem \ref{th:DU_s2_dEven}.
All of these three results rely on a key lemma, which we establish first.
Finally, in Section \ref{sc:data}, we present computational data for the DU of
the Wan-Lidl PPs of the form described in Theorem \ref{th:DU_sEven_d2} over
some prime fields $\ff{p}$, and for some small values of $s$.

\section{The Wan-Lidl PPs}\label{sc:WLPPs}

In \cite{wan91}, Wan and Lidl studied the permutation behaviour of 
polynomials of the form $x^s h(x^{(q-1)/d})$. In particular, they determined
necessary and sufficient conditions for them to be PPs, as well as
establishing results about their group structure under composition modulo
$x^q-x$.
Their classification result is as follows.
\begin{theorem}[\cite{wan91}]\label{th:WanLidl}
Let $s,d\in\nat$ with $d\mid q-1$.
Let $h\in\ffx{q}$, $T(x)=x^{(q-1)/d}$, and $H=\cyc{\gen^{(q-1)/d}}$.
Then $f(x)=x^s h(T(x))$ is a PP over $\ff{q}$ if and only if
\begin{enumerate}[label=(WL \arabic*)]
\item $\gcd (s,\frac{q-1}{d})=1$. 
\item $h(\lambda)\neq 0$ for all $\lambda\in H$.
\item $x^s T(h(x))$ is injective on $H$. 
\end{enumerate}
\end{theorem}
We may assume that $\deg(h)<d$ when studying Wan-Lidl polynomials $f=x^s h(T(x))$. Indeed, since $s>1$, if $h$ has a term $x^d$, then the reduction modulo $x^q-x$ of this term in $f$ is $x^s(x^{(q-1)/d})^d\equiv x^{s+(q-1)} \equiv x^s\pmod{x^q-x}$. 

One important aspect about these polynomials is that their overall
behaviour is tied to their behaviour on the subgroup $H$.
It is for this reason that we were first attracted to studying the DU of
these polynomials.
Restrictions such as this one have been used before in the study of low DU
functions and their bijectiveness.
For example, Budgahyan, Carlet and Leander \cite{bud09} produced a
method for constructing APN functions from known APN functions using a
restriction condition.
In the aforementioned paper of Calderini \cite{cal21}, the author also uses
knowledge about a function's behaviour on a subfield to obtain a construction
of low DU permutations.
Very recently, Bergman and Coulter \cite{coulter22a} used a restriction condition
to prove a class of 4-DU functions were not bijections. 
The proofs for our results follow a similar approach to these previous low DU
results.

\section{Determining the DU for the general case}\label{sc:general}

We shall now prove Theorem \ref{th:DU_general_UB}.
To this end, fix $s,d\in\nat$ with $s>1$ and $d\mid (q-1)$.
Let $h\in\ffx{q}$, $T(x)=x^{(q-1)/d}$, and set $H=\cyc{\gen^{(q-1)/d}}$.
Note that $T$ is a multiplicative function. That is, for
$\alpha,\beta\in\ff{q}$, we have $T(\alpha\beta)=T(\alpha)T(\beta)$.
Additionally, $T$ maps $\ffs{q}$ into $H$, so that the only possible values of
$T(x)$ are in $H\cup\{0\}$. 
Finally, we Let $f(x)=x^s h(T(x))$, which is the Wan-Lidl polynomial whose DU
we wish to determine.

To determine $\delta_f$, we count the number of solutions of $\Delta_{f,a}(x)=c$ for arbitrary $a\in\ffs{q}$ and $c\in\ff{q}$ in the worst case scenario. For $a\in\ffs{q}$, we have 
\begin{equation*}
\Delta_{f,a}(x)=f(x+a)-f(x)=(x+a)^s h\big(T(x+a)\big)-x^s h\big(T(x)\big).
\end{equation*} 
There are four cases to be considered, based on the values of $T(x+a)$ and
$T(x)$.
\begin{itemize}
\item \textbf{Case 1: } If $x\notin\{0,-a\}$, and $(T(x+a),T(x))=(\lambda,\lambda)\in H\times H$, then 
\begin{equation*}
\begin{aligned}
\Delta_{f,a}(x)&=h(\lambda)((x+a)^s-x^s),
\end{aligned}
\end{equation*}
which has degree at most $s-1$. Hence, for arbitrary $c\in\ff{q}$, $\Delta_{f,a}(x)=c$ has at most $s-1$ solutions in this case. 
\item \textbf{Case 2: } If $x\notin\{0,-a\}$, and $(T(x+a),T(x))=(\lambda,\mu)\in H\times H$, where $\mu\neq \lambda$, then 
\begin{equation*}
\begin{aligned}
\Delta_{f,a}(x)&=(x+a)^s h(\lambda)-x^s h(\mu), 
\end{aligned}
\end{equation*}
which has degree at most $s$. Hence, for arbitrary $c\in\ff{q}$, $\Delta_{f,a}(x)=c$ has at most $s$ solutions in this case. 
\item \textbf{Case 3: } If $x=0$, then $\Delta_{f,a}(0)=f(a)-f(0)=f(a)$. 
\item \textbf{Case 4: } If $x=-a$, then $\Delta_{f,a}(-a)=f(0)-f(-a)=-f(-a)$. 
\end{itemize}
There are $d$ possibilities of $\lambda$ in Case 1, and $d(d-1)$ possibilities of the pair $\mu\neq\lambda$ in Case 2. So the contribution from $\ff{q}\setminus\{0,-a\}$ to the number of solutions of $\Delta_{f,a}(x)=c$ is at most 
\begin{equation}\label{eq:general_sol_case1and2}
d(s-1)+d(d-1)s=d(sd-1).
\end{equation}
Moreover, if $c=f(a)$ or $c=-f(-a)$, then Case 3 and Case 4 each gives one solution, respectively. Hence, for arbitrary $c\in\ff{q}$, the number of solutions of $\Delta_{f,a}(x)=c$ is at most $d(sd-1)+2$. This completes the proof of Theorem \ref{th:DU_general_UB}. 

\section{Determining the DU in restricted settings}\label{sc:special}

For the remainder, for distinct $\mu,\lambda\in H$, and fixed $a\in\ffs{q}$ and $c\in\ff{q}$, we call Case 1 in the proof of Theorem \ref{th:DU_general_UB} ``Case $(\lambda,\lambda)$'', and call Case 2 ``Case $(\lambda,\mu)$''. For fixed $a\in\ffs{q}$ and $c\in\ff{q}$, let 
\begin{align}\label{eq:case1_def_g}
g_{\lambda,\lambda}(x)=((x+a)^s-x^s)h(\lambda)-c
\end{align} be the polynomial $\Delta_{f,a}(x)-c$ obtained from Case $(\lambda,\lambda)$, and 
\begin{align}\label{eq:case2_def_g}
g_{\lambda,\mu}(x)=(x+a)^s h(\lambda)-x^s h(\mu)-c 
\end{align}
be the polynomial $\Delta_{f,a}(x)-c$ obtained from Case $(\lambda,\mu)$. In the following lemma, we prove some necessary conditions for
when $g_{\lambda,\lambda}(x)$ (\textit{resp.} $g_{\lambda,\mu}(x)$) has the highest possible degree and splits over $\ff{q}$, and when the roots all satisfy $(T(x+a),T(x))=(\lambda,\lambda)$ (\textit{resp.} $(T(x+a),T(x))=(\lambda,\mu)$). 

\begin{lemma}\label{lm:prod_roots}
Let $s,d\in\nat$ with $s>1$, and $d\mid q-1$. Let $h\in\ffx{q}$, $T(x)=x^{(q-1)/d}$, and $f(x)=x^s h(T(x))$. Let $H=\cyc{\gen^{(q-1)/d}}$.
For fixed $a\in\ffs{q}$, $c\in\ff{q}$, we have the following. 
\begin{enumerate}[label=(\alph*)]
\item (Case $(\lambda,\lambda)$) For $\lambda\in H$, the polynomial $g_{\lambda,\lambda}(x)=((x+a)^s-x^s)h(\lambda)-c$  has degree at most $s-1$. If $g_{\lambda,\lambda}(x)$ has $s-1$ roots in $\ff{q}$, $x_1,x_2,\dots,x_{s-1}$, then 
\begin{equation}\label{eq:prod_roots_case1}
\begin{aligned}
\prod_{i=1}^{s-1}x_i&= \frac{(-1)^{s-1}(a^{s}h(\lambda)-c)}{sa\cdot h(\lambda)}, 
\end{aligned}
\end{equation}
and
\begin{equation}\label{eq:prod_roots+a_case1}
\begin{aligned}
\prod_{i=1}^{s-1}(x_i+a)&= \frac{a^{s}h(\lambda)+(-1)^{s}c}{sa\cdot h(\lambda)}. 
\end{aligned}
\end{equation}
Moreover, if $(T(x_i+a),T(x_i))=(\lambda,\lambda)\in H\times H$ for $1\le i\le s-1$, then 
\begin{equation}\label{eq:T_roots_case1}
\begin{aligned}
\lambda^{s-1}=\frac{T(-1)^{s-1}T(a^s h(\lambda)-c)}{T(sa)T(h(\lambda))}, 
\end{aligned}
\end{equation}
and 
\begin{equation}\label{eq:T_roots+a_case1}
\begin{aligned}
\lambda^{s-1}=\frac{T(a^{s}h(\lambda)+(-1)^{s}c)}{T(sa)T(h(\lambda))}. 
\end{aligned}
\end{equation}

\item (Case $(\lambda,\mu)$) For distinct $\lambda, \mu\in H$, the polynomial $g_{\lambda,\mu}(x)=(x+a)^s h(\lambda)-x^s h(\mu)-c$  has degree at most $s$. If $g_{\lambda,\mu}(x)$ has $s$ roots in $\ff{q}$, $x_1,x_2,\dots,x_{s}$, then 
\begin{equation}\label{eq:prod_roots_case2}
\begin{aligned}
\prod_{i=1}^{s}x_i&= \frac{(-1)^{s}(a^{s}h(\lambda)-c)}{h(\lambda)-h(\mu)}, 
\end{aligned}
\end{equation}
and
\begin{equation}\label{eq:prod_roots+a_case2}
\begin{aligned}
\prod_{i=1}^{s}(x_i+a)=\frac{(-1)(a^{s}h(\mu)+(-1)^{s}c)}{h(\lambda)-h(\mu)}. 
\end{aligned}
\end{equation}
Moreover, if $(T(x_i+a),T(x_i))=(\lambda,\mu)\in H\times H$ for $1\le i\le s$, then 
\begin{equation}\label{eq:T_roots_case2}
\begin{aligned}
\mu^{s}=\frac{T(-1)^{s}T(a^s h(\lambda)-c)}{T(h(\lambda)-h(\mu))}, 
\end{aligned}
\end{equation}
and
\begin{equation}\label{eq:T_roots+a_case2}
\begin{aligned}
\lambda^{s}=\frac{T(-1)T(a^{s}h(\mu)+(-1)^{s}c)}{T(h(\lambda)-h(\mu))}. 
\end{aligned}
\end{equation}
\end{enumerate}
\end{lemma}
\begin{proof}
\begin{enumerate}[label=(\alph*)]
\item[]
\item First, observe that the leading term of $g_{\lambda,\lambda}(x)$ is $sa\cdot h(\lambda)\cdot x^{s-1}$. So $g_{\lambda,\lambda}(x)$ has degree at most $s-1$. This leading term does not vanish if $p\nmid s$ and $h(\lambda)\neq 0$. Suppose this is the case so that $g_{\lambda,\lambda}(x)$ has degree exactly $s-1$. If $g_{\lambda,\lambda}(x)$ has $s-1$ roots $x_i\in\ff{q}$, $1\le i\le s-1$, then we can write 
\begin{align}\label{eq:case1_split_g}
g_{\lambda,\lambda}(x)=sa\cdot h(\lambda)\prod_{i=1}^{s-1}(x-x_i). 
\end{align}
By using both (\ref{eq:case1_def_g}) and (\ref{eq:case1_split_g}) to compute $g_{\lambda,\lambda}(0)$, we have 
\begin{equation}\label{eq:case1_plugin_g0}
\begin{aligned}
a^s h(\lambda)-c&=g_{\lambda,\lambda}(0)\\
&=sa\cdot h(\lambda)\prod_{i=1}^{s-1}(-x_i)=(-1)^{s-1}sa\cdot h(\lambda)\prod_{i=1}^{s-1}x_i. 
\end{aligned}
\end{equation}
Dividing both sides of (\ref{eq:case1_plugin_g0}) by $(-1)^{s-1}sa\cdot h(\lambda)$ gives (\ref{eq:prod_roots_case1}). 

Similarly, by using both (\ref{eq:case1_def_g}) and (\ref{eq:case1_split_g}) to compute $g_{\lambda,\lambda}(-a)$, we have 
\begin{equation*}
\begin{aligned}
-(-a)^s h(\lambda)-c&=g_{\lambda,\lambda}(-a)\\
&=sa\cdot h(\lambda)\prod_{i=1}^{s-1}(-a-x_i)=(-1)^{s-1}sa\cdot h(\lambda)\prod_{i=1}^{s-1}(x_i+a). 
\end{aligned}
\end{equation*}
Therefore, 
\begin{equation*}
\begin{aligned}
\prod_{i=1}^{s-1}(x_i+a)&=\frac{-(-a)^s h(\lambda)-c}{(-1)^{s-1}sa\cdot h(\lambda)}=\frac{(-1)^s((-1)^s a^s h(\lambda)+c)}{sa\cdot h(\lambda)}=\frac{ a^s h(\lambda)+(-1)^s c}{sa\cdot h(\lambda)}, 
\end{aligned}
\end{equation*}
which proves (\ref{eq:prod_roots+a_case1}). \\
Now assume that $(T(x_i+a),T(x_i))=(\lambda,\lambda)\in H\times H$ for $1\le i\le s-1$. Applying $T$ to both sides of (\ref{eq:prod_roots_case1}) and (\ref{eq:prod_roots+a_case1}), and using the fact that $T$ is a multiplicative function, we obtain (\ref{eq:T_roots_case1}) and (\ref{eq:T_roots+a_case1}). 

\item The leading term of $g_{\lambda,\mu}(x)$ is clearly $(h(\lambda)-h(\mu))x^s$, so $g_{\lambda,\mu}(x)$ has degree at most $s$. Suppose $h(\lambda)\neq h(\mu)$ so that $g_{\lambda,\mu}(x)$ has degree exactly $s$. If $g_{\lambda,\mu}(x)$ has $s$ roots $x_i\in\ff{q}$, $1\le i\le s$, then we can write 
\begin{align}\label{eq:case2_split_g}
g_{\lambda,\mu}(x)=(h(\lambda)-h(\mu))\prod_{i=1}^{s}(x-x_i). 
\end{align}
By using both (\ref{eq:case2_def_g}) and (\ref{eq:case2_split_g}) to compute $g_{\lambda,\mu}(0)$, we have 
\begin{equation}\label{eq:case2_plugin_g0}
\begin{aligned}
a^s h(\lambda)-c&=g_{\lambda,\mu}(0)\\
&=(h(\lambda)-h(\mu))\prod_{i=1}^{s}(-x_i)=(-1)^{s}(h(\lambda)-h(\mu))\prod_{i=1}^{s}x_i. 
\end{aligned}
\end{equation}
Dividing both sides of (\ref{eq:case2_plugin_g0}) by $(-1)^{s}(h(\lambda)-h(\mu))$ gives (\ref{eq:prod_roots_case2}).

Similarly, by using both (\ref{eq:case2_def_g}) and (\ref{eq:case2_split_g}) to compute $g_{\lambda,\mu}(-a)$, we have 
\begin{equation*}
\begin{aligned}
-(-a)^s h(\mu)-c&=g_{\lambda,\mu}(-a)\\
&=(h(\lambda)-h(\mu))\prod_{i=1}^{s}(-a-x_i)=(-1)^{s}(h(\lambda)-h(\mu))\prod_{i=1}^{s}(x_i+a). 
\end{aligned}
\end{equation*}
Therefore, 
\begin{equation*}
\begin{aligned}
\prod_{i=1}^{s}(x_i+a)&=\frac{-(-a)^s h(\mu)-c}{(-1)^{s}(h(\lambda)-h(\mu))}=\frac{(-1)^{s+1}((-1)^s a^s h(\mu)+c)}{h(\lambda)-h(\mu)}=\frac{ (-1)(a^s h(\mu)+(-1)^s c)}{h(\lambda)-h(\mu)}, 
\end{aligned}
\end{equation*}
which proves (\ref{eq:prod_roots+a_case2}). 

Finally, suppose that $(T(x_i+a),T(x_i))=(\lambda,\mu)\in H\times H$ for $1\le i\le s-1$. Applying $T$ to both sides of (\ref{eq:prod_roots_case2}) and (\ref{eq:prod_roots+a_case2}), and using the fact that $T$ is a multiplicative function, we obtain (\ref{eq:T_roots_case2}) and (\ref{eq:T_roots+a_case2}). 

\end{enumerate}
\end{proof}

\begin{remark}
Since we are counting the number of solutions of $\Delta_{f,a}(x)=c$ in the worst case scenario, we need not consider the situation where $\deg(g_{\lambda,\lambda})<s-1$ or $\deg(g_{\lambda,\mu})<s$, i.e., when the leading term vanishes in Case $(\lambda,\lambda)$ or Case $(\lambda,\mu)$, respectively. Besides, in most of the cases that we are interested in, we may assume that $p\nmid s$ and $h(\lambda)\neq 0$ so that $\deg(g_{\lambda,\lambda})=s-1$, so that we are mostly interested in this case anyway.
\end{remark}
This technical lemma will form the basis of all of our remaining results.
The proofs of Theorem \ref{th:DU_sEven_d2} and \ref{th:DU_s2_dEven} are based on the following framework. First, (\ref{eq:general_sol_case1and2}) shows that for any $c\in\ff{q}$ and $a\in\ffs{q}$, $\Delta_{f,a}(x)=c$ has at most $d(sd-1)$ solutions in $\ff{q}\setminus\{0,-a\}$. We can reduce this number by the arguments in the next few paragraphs. Second, since $\Delta_{f,a}(0)=f(a)$, we check if the upper bound of the number of solutions in $\ff{q}\setminus\{0,-a\}$ found in the first step can be reduced further for $c=f(a)$. Otherwise, the bound of DU goes up by $1$ to include $x=0$ when $c=f(a)$. Finally, since $\Delta_{f,a}(-a)=-f(-a)$, we need to repeat the previous process for $c=-f(-a)$. 

If $f$ is a PP, then by Theorem \ref{th:WanLidl} (WL 3), $\lambda^{s}T( h(\lambda))\neq \mu^{s}T( h(\mu))$ whenever $\mu, \lambda\in H$, $\mu\neq \lambda$. In fact, when all the assumptions of both Case $(\lambda,\lambda)$ and Case $(\lambda,\mu)$ in Lemma \ref{lm:prod_roots} hold, i.e., they contribute the maximum number of solutions to $\Delta_{f,a}(x)=c$, we can obtain an expression of $\lambda^{s}T( h(\lambda))$ by combining (\ref{eq:T_roots_case1}) and (\ref{eq:T_roots_case2}) as follows. First, from (\ref{eq:T_roots_case2}) we can solve 
\begin{equation*}
\begin{aligned}
T(a^s h(\lambda)-c)=\frac{\mu^{s}T(h(\lambda)-h(\mu))}{T(-1)^{s}}. 
\end{aligned}
\end{equation*}
Substituting $T(a^s h(\lambda)-c)$ into (\ref{eq:T_roots_case1}) gives
\begin{equation*}
\begin{aligned}
\lambda^{s-1}&=\frac{T(-1)^{s-1}\mu^{s}T(h(\lambda)-h(\mu))}{T(-1)^{s}T(sa)T( h(\lambda))}=\frac{\mu^{s}T(h(\lambda)-h(\mu))}{T(-1)T(sa)T( h(\lambda))}. 
\end{aligned}
\end{equation*}
Multiplying both sides by $\lambda T( h(\lambda))$ yields
\begin{equation*}
\begin{aligned}
\lambda^{s}T( h(\lambda))&=C_a\lambda\mu^{s}T(h(\lambda)-h(\mu)), 
\end{aligned}
\end{equation*}
where $C_a=1/(T(-1)T(sa))$ is a constant that does not depend on $\mu$ and $\lambda$. Similarly, if the assumptions of both Case $(\mu,\mu)$ and Case $(\mu,\lambda)$ are satisfied, we have
\begin{equation*}
\begin{aligned}
\mu^{s}T( h(\mu))&=C_a\mu\lambda^{s}T(h(\mu)-h(\lambda))\\
&= T(-1)C_a\mu\lambda^{s}T(h(\lambda)-h(\mu)). 
\end{aligned}
\end{equation*}
Note that these expressions are independent of $c$. \\

If $T(-1)=(\mu/\lambda)^{s-1}$, then 
\begin{equation}\label{eq:mu_lambda_equal}
\mu^{s}T( h(\mu))=(\mu/\lambda)^{s-1}C_a\mu\lambda^{s}T(h(\lambda)-h(\mu))=C_a\lambda\mu^{s}T(h(\lambda)-h(\mu))=\lambda^{s}T( h(\lambda)), 
\end{equation}
which contradicts Theorem \ref{th:WanLidl} (WL 3). Therefore, if $T(-1)=(\mu/\lambda)^{s-1}$, the equation $\Delta_{f,a}(x)=c$ cannot have the maximum number of solutions simultaneously in Case $(\lambda,\lambda)$, Case $(\mu,\mu)$, Case $(\lambda,\mu)$, and Case $(\mu,\lambda)$. Hence, the number of solutions of $\Delta_{f,a}(x)=c$ that are contributed by these four cases must be strictly less than $2s+2(s-1)=4s-2$. 

For the remaining, we use this method to prove improved bounds of the DU for
some special cases when $T(-1)=(\mu/\lambda)^{s-1}$. Since $T(-1)\in\{\pm 1\}$, one possible future direction is to completely investigate all cases whose bound of DU can be improved by the aforementioned method. 

\subsection{Proof of Theorem \ref{th:DU_sEven_d2}}
In this subsection, we consider the case where $d=2$ and $s$ is even.
Take $h\in\ffx{q}$, $T(x)=x^{(q-1)/2}$, and $H=\cyc{\gen^{(q-1)/2}}=\{\pm 1\}$. Note that now $T$ is the quadratic character $\eta$ of $\ff{q}$, so we use $\eta$ instead of $T$ for the remaining of this subsection. Let $f(x)=x^s h(\eta(x))$. First, a few things can be simplified as follows. 
\begin{enumerate}[label=(\roman*)]
\item Since $d=2$ and we are not interested in monomial PPs, we may assume that $h(x)=x+b$ for some $b\in\ffs{q}$. Moreover, the three necessary and sufficient conditions of Theorem \ref{th:WanLidl} give the following restrictions on $q$ and $b$ for $f(x)=x^s(\eta(x)+b)$ being a PP. 
\begin{enumerate}[label=(WL \arabic*)]
\item $\gcd (s,\frac{q-1}{2})=1$: Since $s$ is even, $(q-1)/2$ must be odd. So $q\equiv 3\bmod 4$ and $\eta(-1)=-1$. 
\item $h(\lambda)\neq 0$ for all $\lambda\in H$: This gives $(b\pm 1)\neq 0$ so $b\neq \pm 1$. 
\item $x^s T(h(x))$ is injective on $H$: This gives $\eta(b+1)\neq \eta(b-1)$. So $\eta(b+1)=-\eta(b-1)$. 
\end{enumerate}
%\item We may always assume that $p\nmid s$, which is required to apply Lemma \ref{lm:prod_roots} (a). Because if $f=x^{ps}(\eta(x)+b)$ in $\ffx{p^e}$, we can compose $f$ with $x^{p^{e-1}}$ and obtain that
%\begin{equation}
%f\Big(x^{p^{e-1}}\Big)=\Big(x^{p^{e-1}}\Big)^{ps}\Big(\eta\Big(x^{p^{e-1}}\Big)+b\Big)=x^s(\eta(x)+b). 
%\end{equation}
%Since $x^{p^{e-1}}$ is a PP over $\ffx{p^e}$, $f(x^{p^{e-1}})$ is linearly equivalent to $f$ and therefore has the same DU. \\
%Furthermore, since $H=\{\pm 1\}$, $h(1)=b+1\neq b-1=h(-1)$, the requirement $h(\lambda)\neq h(\mu)$ for $\lambda\neq \mu\in H$ in Lemma \ref{lm:prod_roots} (b) also holds. 
\item For a given $b\in\ffs{q}$, $f(x)=x^s(\eta(x)+b)$ and $f'(x)=x^s(\eta(x)-b)$ are linearly equivalent, since 
\begin{equation*}
\begin{aligned}
-f(-x)=-(-x)^s(\eta(-x)+b)=-x^s(-\eta(x)+b)=x^s(\eta(x)-b)=f'(x). 
\end{aligned}
\end{equation*}

Moreover, we claim that we only need to check $\Delta_{f,1}(x)$ for determining the DU. First, by the fact that $f'(x)=-f(-x)$ and setting $y=-x-1$, we have $\Delta_{f',1}(x)=-f(-(x+1))-(-f(-x))=-f(y)+f(y+1)=\Delta_{f,1}(y)$. Hence, for all $c\in \ff{q}$, the number of solutions of $\Delta_{f,1}(x)=c$ is the same as $\Delta_{f',1}(x)=c$. Next, fix $a\in\ffs{q}$ and $c\in\ff{q}$. Then 
\begin{equation*}
\begin{aligned}
c=\Delta_{f,a}(x)&=(x+a)^s(\eta(x+a)+b)-x^s(\eta(x)+b)\\
&=a^s\left(\frac{x}{a}+1\right)^s\left(\eta\left(a\left(\frac{x}{a}+1\right)\right)+b\right)-a^s\frac{x^s}{a^s}\left(\eta\left(a\cdot\frac{x}{a}\right)+b\right).
\end{aligned}
\end{equation*}
Substituting $Y=\frac{x}{a}$ and dividing both sides by $a^s\eta(a)$ yields
\begin{equation*}
\begin{aligned}
\left(Y+1\right)^s\left(\eta\left(Y+1\right)+\frac{b}{\eta(a)}\right)-Y^s\left(\eta\left(Y\right)+\frac{b}{\eta(a)}\right)=\frac{c}{a^s\eta(a)}. 
\end{aligned}
\end{equation*}
The left hand side equals to either $\Delta_{f,1}(Y)$ or $\Delta_{f',1}(Y)$, depending on the value of $\eta(a)$. And for fixed $a\in\ffs{q}$, as $c$ runs over all $\ff{q}$, so does $c'=\frac{c}{a^s\eta(a)}$. Therefore, 
\[\delta_{f}=\max_{a\in\ffs{q},c\in\ff{q}}|\{x\in\ff{q}\mid \Delta_{f,a}(x)=c\}|=\max_{a\in\ffs{q},c'\in\ff{q}}|\{x\in\ff{q}\mid \Delta_{f,1}(x)=c'\}|.\]
\end{enumerate}

To prove Theorem \ref{th:DU_sEven_d2}, we first show that $\Delta_{f,1}(x)=c$ has at most $4s-3$ solutions in $\ff{q}\setminus\{0,-1\}$ for any $c\in\ff{q}$. From (\ref{eq:general_sol_case1and2}) we know that $\Delta_{f,1}(x)=c$ has at most $4s-2$ solutions in $\ff{q}\setminus\{0,-1\}$ for any $c\in\ff{q}$, where the contribution came from Case $(1,1)$, Case $(-1,-1)$, Case $(1,-1)$, and Case $(-1,1)$ since $H=\{\pm 1\}$. By (\ref{eq:mu_lambda_equal}), since $-1=\eta(-1)=(\mu/\lambda)^{s-1}=-1$ as $\{\mu,\lambda\}=\{\pm 1\}$, these four cases cannot have the maximum number of solutions simultaneously. So for all $c\in\ff{q}$, $\Delta_{f,1}(x)=c$ has at most $4s-3$ solutions in $\ff{q}\setminus\{0,-1\}$. 

Next, let $c=f(1)=b+1$ so that $0$ is a solution of $\Delta_{f,1}(x)=c$. First note that if $-1$ was also a solution, then $c=-f(-1)=-b+1$, which means $b=0$. So $-1$ cannot be another solution as long as $b\in\ffs{q}$. 

We need to check the number of solutions in $\ff{q}\setminus\{0,-1\}$. By Lemma \ref{lm:prod_roots} (a), if $g_{\lambda,\lambda}(x)$ has $s-1$ roots, $x_1,\dots,x_{s-1}\in\ff{q}$, then

\begin{equation*}
\begin{aligned}
\prod_{i=1}^{s-1}x_i&= \frac{(-1)^{s-1}(h(\lambda)-c)}{s\cdot h(\lambda)}=\frac{-(b+\lambda-b-1)}{s\cdot (b+\lambda)}=\frac{1-\lambda}{s\cdot (b+\lambda)}.
\end{aligned}
\end{equation*}
If $x_i\neq 0$ for all $1\le i\le s-1$, then $\prod_{i=1}^{s-1}x_i\neq 0$ and $\lambda\neq 1$. This means that only Case $(-1,-1)$ may contribute $s-1$ solutions in $\ff{q}\setminus\{0,-1\}$, but Case $(1,1)$ has at most $s-2$ solutions in $\ff{q}\setminus\{0,-1\}$. 

Similarly, by Lemma \ref{lm:prod_roots} (b), if $g_{\lambda,\mu}(x)$ has $s$ roots, $x_1,\dots,x_{s}\in\ff{q}$, then
\begin{equation*}
\begin{aligned}
\prod_{i=1}^{s}x_i&= \frac{(-1)^{s}(h(\lambda)-c)}{h(\lambda)-h(\mu)}=\frac{b+\lambda-b-1}{b+\lambda-b-\mu}=\frac{\lambda-1}{2\lambda}, \text{ since }\mu=-\lambda. 
\end{aligned}
\end{equation*}
If $x_i\neq 0$ for all $1\le i\le s$, then $\prod_{i=1}^{s}x_i\neq 0$ and $\lambda\neq 1$. This means that only Case $(-1,1)$ may have $s$ solutions in $\ff{q}\setminus\{0,-1\}$, but Case $(1,-1)$ has at most $s-1$ solutions in $\ff{q}\setminus\{0,-1\}$. Hence, when $c=f(1)$, $\Delta_{f,1}(x)=c$ has at most $1+(s-1)+(s-2)+s+(s-1)=4s-3$ solutions. 

Lastly, let $c=-f(-1)=1-b$ so that $-1$ is a solution of $\Delta_{f,1}(x)=c$. We have shown that $0$ cannot be another solution if $b\neq 0$. Again, by Lemma \ref{lm:prod_roots} (a), if $g_{\lambda,\lambda}(x)$ has $s-1$ roots, $x_1,\dots,x_{s-1}$ in $\ff{q}$, then 
\begin{equation*}
\begin{aligned}
\prod_{i=1}^{s-1}(x_i+1)&= \frac{h(\lambda)+(-1)^{s}c}{s\cdot h(\lambda)}=\frac{b+\lambda+1-b}{s(b+\lambda)}=\frac{\lambda+1}{s(b+\lambda)}. 
\end{aligned}
\end{equation*}
If $x_i\neq -1$ for all $1\le i\le s-1$, then $\prod_{i=1}^{s-1}(x_i+1)\neq 0$ and $\lambda\neq -1$. Therefore, only Case $(1,1)$ may have $s-1$ solutions in $\ff{q}\setminus\{0,-1\}$, but Case $(-1,-1)$ has at most $s-2$ solutions in $\ff{q}\setminus\{0,-1\}$. 

Similarly, by Lemma \ref{lm:prod_roots} (b), if $g_{\lambda,\mu}(x)$ has $s$ roots, $x_1,\dots,x_{s}$ in $\ff{q}$, then
\begin{equation*}
\begin{aligned}
\prod_{i=1}^{s}(x_i+1)=\frac{(-1)(h(\mu)+(-1)^{s}c)}{h(\lambda)-h(\mu)}=\frac{(-1)(b+\mu+1-b)}{2\lambda}=\frac{-(-\lambda+1)}{2\lambda}, \text{ since }\mu=-\lambda. 
\end{aligned}
\end{equation*}
If $x_i\neq -1$ for all $1\le i\le s$, then $\prod_{i=1}^{s}(x_i+1)\neq 0$ and $\lambda\neq 1$. Therefore, only Case $(-1,1)$ may have $s$ solutions in $\ff{q}\setminus\{0,-1\}$, but Case $(1,-1)$ has at most $s-1$ solutions in $\ff{q}\setminus\{0,-1\}$. Hence, when $c=-f(-1)$, $\Delta_{f,1}(x)=c$ has at most $4s-3$ solutions. 

We conclude that when $d=2$ and $s$ is even, $\Delta_{f,1}(x)=c$ has at most $4s-3$ solutions for any $c\in \ff{q}$. This completes the proof of Theorem \ref{th:DU_sEven_d2}. 

\subsection{Proof of Corollary \ref{co:DU_s2d2b3}} Corollary \ref{co:DU_s2d2b3} is a special case of Theorem \ref{th:DU_sEven_d2} when $s=2$, $b=\pm 3$, and $q\equiv 3\pmod 8$. Because $x^2(\eta(x)+3)$ and $x^2(\eta(x)-3)$ are linearly equivalent, we only need to prove the result for $b=3$. Since $q\equiv 3\pmod 8$, $\eta(2)=-1$ and $\gcd(2,(q-1)/2)=1$. Also, $\eta(3+1)=1$ and $\eta(3-1)=-1$. So $f(x)$ is always a PP by Theorem \ref{th:WanLidl} under these settings. 

Since $H=\{\pm 1\}$, there are only four cases when $x\in \ff{q}\setminus\{0,-1\}$: Case $(1,1)$, Case $(-1,-1)$, Case $(1,-1)$, and Case $(-1,1)$. First, since $s=2$, $g_{\lambda,-\lambda}(x)$ is quadratic:
\begin{equation*}
\begin{aligned}
\Delta_{f,1}(x)-c&=(x+1)^2(b+\lambda)-x^2(b-\lambda)-c\\
&=2\Big(\lambda x^2+(b+\lambda)x+(b+\lambda-c)/2\Big). 
\end{aligned}
\end{equation*}
If $\Delta_{f,1}(x)-c$ has two roots, the discriminant $(b+\lambda)^2-2\lambda(b+\lambda-c)$ should be either a square in $\ff{q}$ if the two roots are distinct, or should be $0$ if the roots coincide. For Case $(1,-1)$, the discriminant is 
\begin{equation*}
\begin{aligned}
D=(b+1)^2-2(b+1-c)=(b+1)^2+2(c-b-1),
\end{aligned}
\end{equation*}
and for Case $(-1,1)$, the discriminant is 
\begin{equation*}
\begin{aligned}
D'=(b-1)^2+2(b-1-c)=(b-1)^2-2(c-b+1). 
\end{aligned}
\end{equation*}

Now set $b=3$ in $D$ and $D'$, and together with simplified (\ref{eq:T_roots_case1}), (\ref{eq:T_roots+a_case1}), (\ref{eq:T_roots_case2}), and (\ref{eq:T_roots+a_case2}) about $\eta$ values for $s=2$ and $a=1$ from Lemma \ref{lm:prod_roots}, we obtain the following. 

\begin{itemize}
\item {\bf Case $(1,1)$:}  If $\Delta_{f,1}(x)=c$ has $1$ solution in Case $(1,1)$, then 
\begin{equation}\label{eq:b3case1}
\begin{aligned}
\eta(c-4)=-1,\qquad\eta(c+4)=-1. 
\end{aligned}
\end{equation}

\item {\bf Case $(-1,-1)$:}  If $\Delta_{f,1}(x)=c$ has $1$ solution in Case $(-1,-1)$, then 
\begin{equation}\label{eq:b3case2}
\begin{aligned}
\eta(c-2)=-1,\qquad\eta(c+2)=-1. 
\end{aligned}
\end{equation}

\item {\bf Case $(1,-1)$:}  If $\Delta_{f,1}(x)=c$ has $2$ solutions in Case $(1,-1)$, $x_1,x_2$, then 
\begin{equation}\label{eq:b3case3_hasSol}
\begin{aligned}
\eta(c-4)=1,\qquad\eta(c+2)=1.
\end{aligned}
\end{equation}
Moreover, since 
\begin{equation*}
\begin{aligned}
D=4^2+2(c-4)=2(c+4),
\end{aligned}
\end{equation*}
if $x_1\neq x_2$, then 
\begin{equation*}
\begin{aligned}
\eta(c+4)=\eta(2)=-1; 
\end{aligned}
\end{equation*}
otherwise, if $x_1=x_2$, then 
\begin{equation*}
\begin{aligned}
c=-4. 
\end{aligned}
\end{equation*}

\item {\bf Case $(-1,1)$:}  If $\Delta_{f,1}(x)=c$ has $2$ solutions in Case $(-1,1)$, $x_1,x_2$, then 
\begin{equation*}
\begin{aligned}
\eta(c-2)=-1,\qquad\eta(c+4)=-1. 
\end{aligned}
\end{equation*}
Moreover, since 
\begin{equation*}
\begin{aligned}
D'=2^2-2(c-2)=-2(c-4), 
\end{aligned}
\end{equation*}
if $x_1\neq x_2$, then 
\begin{equation}\label{eq:b3case4_2diffSol}
\begin{aligned}
\eta(c-4)=\eta(-2)=1; 
\end{aligned}
\end{equation}
otherwise, if $x_1=x_2$, then 
\begin{equation}\label{eq:b3case4_2sameSol}
\begin{aligned}
c=4. 
\end{aligned}
\end{equation}
\end{itemize}

From these conditions, we see that (\ref{eq:b3case1}) contradicts with (\ref{eq:b3case4_2diffSol}) and (\ref{eq:b3case4_2sameSol}). So if there is any solution from Case $(1,1)$ or Case $(-1,1)$, either Case $(1,1)$ has $1$ solution and Case $(-1,1)$ has no solution, or Case $(1,1)$ has no solution and Case $(-1,1)$ has $1$ or $2$ solutions. Moreover, condition (\ref{eq:b3case2}) contradicts with (\ref{eq:b3case3_hasSol}). So if there is any solutions from Case $(-1,-1)$ or Case $(1,-1)$, either Case $(-1,-1)$ has $1$ solution and Case $(1,-1)$ has at most $1$ solution, or Case $(-1,-1)$ has no solution and Case $(1,-1)$ has $1$ or $2$ solutions. 

To sum up, we have shown that there are no more than $2$ solutions in total from Case $(1,1)$ and Case $(-1,1)$, and no more than $2$ solutions in total from Case $(-1,-1)$ and Case $(1,-1)$. Hence, $\Delta_{f,1}(x)=c$ has at most $4$ solutions in $\ff{q}\setminus\{0,-1\}$ for any $c\in\ff{q}$. Finally, if $x=0$ is a solution of $\Delta_{f,1}(x)=c$, then $c=f(1)=4$. This cannot satisfy (\ref{eq:b3case1}), (\ref{eq:b3case3_hasSol}), and (\ref{eq:b3case4_2diffSol}). This means that there is no solution from Case $(1,1)$, and there is at most $1$ solution from each of Case $(-1,-1)$, Case $(1,-1)$, and Case $(-1,1)$. On the other hand, if $x=-1$ is a solution, then $c=-f(-1)=-2$. This cannot satisfy (\ref{eq:b3case2}) and (\ref{eq:b3case3_hasSol}). So there is no solution from Case $(-1,-1)$, and at most $1$ solution from Case $(1,-1)$. Moreover, since (\ref{eq:b3case1}) contradicts to (\ref{eq:b3case4_2diffSol}), it is impossible to have $3$ different solutions together from Case $(1,1)$ and Case $(-1,1)$. Thus, there are at most $4$ solutions when $c=f(1)$ or $c=-f(-1)$ as well. 

Therefore, we conclude that $\Delta_{f,1}(x)=c$ has at most $4$ solutions in $\ff{q}$ for any $c\in\ff{q}$, which means the DU of $f(x)=x^2(\eta(x)\pm 3)$ is at most $4$. 

\subsection{Proof of Theorem \ref{th:DU_s2_dEven}} In this subsection, we set $s=2$ and let $d\in\nat$ be even. Since we require that $d\mid q-1$, $q$ must be odd. Let $h\in\ffx{q}$, $T(x)=x^{(q-1)/d}$, and $H=\cyc{\gen^{(q-1)/d}}$. Let $f(x)=x^2 h(T(x))$ be a PP. By Theorem \ref{th:WanLidl} (WL 1), if $f$ is a PP, then $\gcd (2,\frac{q-1}{d})=1$. Hence, $(q-1)/d$ is odd, and $T(-1)=-1$.

To prove Theorem \ref{th:DU_s2_dEven}, first by (\ref{eq:general_sol_case1and2}) we know that $\Delta_{f,a}(x)=c$ has at most $2d^2-d$ solutions in $\ff{q}\setminus\{0,-a\}$ for any $a,c\in\ff{q}$, $a\neq 0$. We shall reduce this count by finding suitable $\lambda,\mu\in H$ that give a contradiction as described in (\ref{eq:mu_lambda_equal}). Since $d$ is even, $-1\in H$. So for $\lambda\in H$, there exists $\mu\in H$ such that $\mu=-\lambda$. This gives the desired condition $-1=T(-1)=(\mu/\lambda)^{s-1}=-1$ that leads to the contradiction in (\ref{eq:mu_lambda_equal}). Therefore, for every pair of $\pm \lambda\in H$, Case $(\lambda,\lambda)$, Case $(-\lambda,-\lambda)$, Case $(\lambda,-\lambda)$, and Case $(-\lambda,\lambda)$ cannot all have the maximum number of solutions simultaneously. Since there are $d/2$ such pairs, we conclude that for any $c\in\ff{q}$, $\Delta_{f,a}(x)=c$ has at most $2d^2-d-\frac{d}{2}=2d^2-\frac{3d}{2}$ solutions in $\ff{q}\setminus\{0,-a\}$. 

Next, let $c=f(a)=a^2 h(T(a))$ so that $x=0$ is a solution of $\Delta_{f,a}(x)=c$. If $x=-a$ is another solution, then $a^2 h(T(a))=c=-f(-a)=-a^2 h(T(-a))$. So if $h(T(-a))=-h(T(a))$, then we need to add one more solution to the worst case. For solutions in $\ff{q}\setminus\{0,-a\}$, since $q$ is odd and $s=2\neq 0$ in $\ff{q}$, $g_{\lambda,\lambda}(x)$ is linear. By Lemma \ref{lm:prod_roots} (a), the root of $g_{\lambda,\lambda}(x)$ is 
\begin{equation*}
\begin{aligned}
x&= \frac{(-1)^{s-1}(a^{s}h(\lambda)-c)}{sa\cdot h(\lambda)}=\frac{-a^2 (h(\lambda)-h(T(a)))}{2a\cdot h(\lambda)}. 
\end{aligned}
\end{equation*}
If $x\neq 0$, then $h(\lambda)-h(T(a))\neq 0$, which implies $\lambda\neq T(a)$. So for fixed $a\in\ffs{q}$, Case $(\lambda,\lambda)$ may give one solution in $\ff{q}\setminus\{0,-a\}$ only if $\lambda\neq T(a)$. This means we have at most $d-1$ solutions in $\ff{q}\setminus\{0,-a\}$ from this type. 

Similarly, by Lemma \ref{lm:prod_roots} (b), if $g_{\lambda,\mu}(x)$ has two roots, $x_1,x_{2}\in \ff{q}$, then
\begin{equation*}
\begin{aligned}
x_1 x_2&= \frac{(-1)^{s}(a^{s}h(\lambda)-c)}{h(\lambda)-h(\mu)}=\frac{a^2 (h(\lambda)-h(T(a)))}{h(\lambda)-h(\mu)}.  
\end{aligned}
\end{equation*}
If $x_1, x_2\neq 0$, then $h(\lambda)-h(T(a))\neq 0$, which implies $\lambda\neq T(a)$. So for fixed $a\in\ffs{q}$, Case $(\lambda,\mu)$ may give two solutions in $\ff{q}\setminus\{0,-a\}$ if $\lambda\neq T(a)$, but at most one solution in $\ff{q}\setminus\{0,-a\}$ if $\lambda=T(a)$. There are $(d-1)^2$ choices of $(\lambda,\mu)$ in the former case, and $d-1$ choices of $(T(a),\mu)$ in the later case. Hence, we have at most $2(d-1)^2+(d-1)=2d^2-3d+1$ solutions in $\ff{q}\setminus\{0,-a\}$ from this type. Therefore, we conclude that when $c=f(a)$, $\Delta_{f,a}(x)=c$ has at most $1+1+(d-1)+(2d^2-3d+1)=2d^2-2d+2$ solutions. 

Finally, let $c=-f(-a)=-a^2 h(T(-a))$ so that $x=-a$ is a solution of $\Delta_{f,a}(x)=c$. For the same reason discussed above, $0$ may be another solution if $h(T(-a))=-h(T(a))$. Next, we check the solutions in $\ff{q}\setminus\{0,-a\}$. Once again by Lemma \ref{lm:prod_roots} (a), the root of $g_{\lambda,\lambda}(x)$ satisfies 
\begin{equation*}
\begin{aligned}
x+a&= \frac{a^{s}h(\lambda)+(-1)^{s}c}{sa\cdot h(\lambda)}=\frac{a^2 (h(\lambda)-h(T(-a)))}{2a\cdot h(\lambda)}. 
\end{aligned}
\end{equation*}
If $x\neq -a$, then $h(\lambda)-h(T(-a))\neq 0$, which implies $\lambda\neq T(-a)$. So for fixed $a\in\ffs{q}$, Case $(\lambda,\lambda)$ may give one solution in $\ff{q}\setminus\{0,-a\}$ only if $\lambda\neq T(-a)$. Similarly, by Lemma \ref{lm:prod_roots} (b), if $g_{\lambda,\mu}(x)$ has two roots, $x_1,x_{2}\in \ff{q}$, then
\begin{equation*}
\begin{aligned}
(x_1+a)(x_2+a)&= \frac{(-1)(a^{s}h(\mu)+(-1)^{s}c)}{h(\lambda)-h(\mu)}=\frac{-a^2 (h(\mu)-h(T(-a)))}{h(\lambda)-h(\mu)}.  
\end{aligned}
\end{equation*}
If $x_1, x_2\neq -a$, then $h(\mu)-h(T(-a))\neq 0$, which implies $\mu\neq T(-a)$. So for fixed $a\in\ffs{q}$, Case $(\lambda,\mu)$ may give two solutions in $\ff{q}\setminus\{0,-a\}$ if $\mu\neq T(-a)$, but at most one solution in $\ff{q}\setminus\{0,-a\}$ if $\mu=T(-a)$. Hence, with a similar count to that made for $c=f(a)$, we conclude that when $c=-f(-a)$, $\Delta_{f,a}(x)=c$ has at most $2d^2-2d+2$ solutions as well. 

We have proved that $\Delta_{f,a}(x)=c$ has at most $2d^2-\frac{3d}{2}$ solutions when $c\notin\{f(a),-f(-a)\}$, and at most $2d^2-2d+2$ solutions when $c\in\{f(a),-f(-a)\}$. It is easy to check that $2d^2-\frac{3d}{2}\ge 2d^2-2d+2$ if and only if $d\ge 4$. If $d=2$, then $2d^2-\frac{3d}{2}=5$, and we are in the situation of Theorem \ref{th:DU_sEven_d2} so that it suffices to consider $a=1$. We have shown in Theorem \ref{th:DU_sEven_d2} that as long as $f$ is not a monomial, $\delta_f\le 5$. If $f$ is a monomial, then $f=x^2\eta(x)$. If $0,-1$ are both solutions of $\Delta_{f,1}(x)=c$, then $c=f(1)=1$. Then $g_{1,1}(x)=(x+1)^2-x^2-1=2x$, so the root of $g_{1,1}(x)$ is $0$. And $g_{1,-1}(x)=(x+1)^2+x^2-1=2x^2+2x=2x(x+1)$, so the two roots of $g_{1,-1}(x)$ are $0$ and $-1$. This means that $\Delta_{f,1}(x)=c$ has at most $3$ solutions in $\ff{q}\setminus\{0,-1\}$, and therefore we still have $\delta_f\le 5$ in this case. This concludes the proof of Theorem \ref{th:DU_s2_dEven}.

\section{Computation Data for Theorem \ref{th:DU_sEven_d2}}\label{sc:data}

Using the Magma algebra system \cite{magma}, we computed the DU of all PPs of
the form $f(x)=x^s(\eta(x)+b)$ described in in Theorem \ref{th:DU_sEven_d2} over some prime fields $\ff{p}$ for $s=2,4,6$. We provide a selection of the
computational results in Table \ref{tb:s2}, Table \ref{tb:s4}, and Table \ref{tb:s6}, respectively. The rows are indexed by $p$, the order of the field, and the columns are indexed by $\delta_f$. The number in row $p$ and column $\delta_f$ represents the number of such $f\in\ffx{p}$ with that exact $\delta_f$. Since $b\neq \pm 1$, and $x^s(\eta(x)+b)$ and $x^s(\eta(x)-b)$ are linearly equivalent, we only test $2\le b\le (p-1)/2$. So the numbers in our tables are actually half of the total counts, if one were to consider all possible $b\in\ffs{p}$. Moreover, recall that Theorem \ref{th:WanLidl} (WL 1) states that a necessary condition for $f$ to be a PP is that $\gcd(s,(p-1)/2)=1$, so we only test $\ff{p}$ that satisfy this condition for a given $s$. For $s=2$ and $4$, this is simply requiring that $p\equiv 3\pmod 4$. For $s=6$, we also need $p\equiv 5\pmod 6$. 

When $s=2$, the bound in Theorem \ref{th:DU_sEven_d2} gives $\delta_f\le 5$.
We computed the DU of all PPs of the form $f(x)=x^2(\eta(x)+b)\in\ffx{p}$ for
all prime fields of order $p<7000$.
We found that there exists $f\in\ffx{p}$ with $\delta_f=5$ for $p=31$ and
$59\le p < 7000$, i.e., almost all tested fields have examples that meet the
upper bound.
Moreover, when $p$ is large (roughly $>1400$), about $95\%$ or more of the
PPs involved have $\delta_f=5$.
Finally, we observe that when $p>4007$, all such $f$ have $\delta_f=5$, except
for when $2$ is a non-square.
In fact, the only example of $\delta_f=4$ found shown in those fields is when
$b=3$, which corresponds to our Corollary \ref{co:DU_s2d2b3}. 

When $s=4$, we do the computation for all prime fields of order $p<10000$.
Our bound from Theorem \ref{th:DU_sEven_d2} gives $\delta_f\le 13$.
In stark contrast to the $s=2$ case, here we found only one example of a PP
$f$ that met the bound; specifically, $f(x)=x^4(\eta(x)+1734)$ over
$\ff{3671}$. For this polynomial, one can check that
$|\{x\in \ff{3671}\mid \Delta_{f,1}(x)=2307\}|=13$.
From Table \ref{tb:s4}, it is easily observed that this example is an outlier.
In fact, there are no examples at all of $\delta_f=12$, and very few examples
of $\delta_f=11$ were found (in just $28$ of the $618$ choices of $p$ in our range).
It can also be seen that although the distribution of $\delta_f$ values does
move up as $p$ grows, it does not keep moving towards the upper bound $13$.
Instead, when $p$ is large, the distribution of $\delta_f$ concentrates around
$\delta_f=6$ and $7$.
When $p>5000$, at least $88\%$ of the $f$ have $\delta_f=6$ or $7$.
This suggests that when $p$ is large, even though the upper bound is $13$,
there is a high probability that a randomly chosen PP of the form of $f$ will
have DU only $6$ or $7$. No PPs with a DU of 4 were found over fields of order $p>3323$. 

When $s=6$, we also do the computation for all prime fields of order $p<10000$.
The bound in Theorem \ref{th:DU_sEven_d2} for this case is $\delta_f\le 21$.
Here, we observed no PP example that approached the bound.
Indeed, the largest DU that we observed is only $\delta_f=13$ over $\ff{5903}$.
There are no examples of $\delta_f\ge 14$ for all tested fields $\ff{p}$.
On this evidence, we highly suspect that the bound can be improved, possibly
significantly, for $s\ge 6$. This is not altogether surprising, as 
as $p$ increases it seems more and more unlikely that the worst-case scenarios
that yield our upperbound could all occur at once.
Additionally, as observed in the $s=4$ data, the distribution of $\delta_f$
also concentrates around $\delta_f=6$ and $7$ when $p$ is large.
Although the distribution does not shift towards these values as quickly as in
the $s=4$ case, we still observe that many fields of size $p>5000$ have more
than $90\%$ of the $f$ with $\delta_f=6$ or $7$. No PPs with a DU of 4 were found over fields of order $p>2579$.

\begin{table}[h]
\caption{\textbf{$s=2$ in Theorem \ref{th:DU_sEven_d2}:} The number of PPs $f$ that have each possible $\delta_f$, where $f$ is of the form $f(x)=x^2(\eta(x)+b)\in\ffx{p}$, $2\le b\le (p-1)/2$. }\label{tb:s2}
\begin{tabular}{|c||*{4}{c|}}\hline
\diagbox[height=6mm]{$p$}{$\delta_f$}	& 2 &	3	&	4	&	5			\\\hline
7	&	0	&	1	&	0	&	0		\\\hline
11	&	0	&	2	&	0	&	0		\\\hline
19	&	1	&	1	&	2	&	0		\\\hline
23	&	1	&	3	&	1	&	0		\\\hline
31	&	1	&	1	&	4	&	1		\\\hline
43	&	0	&	5	&	5	&	0		\\\hline
47	&	0	&	5	&	6	&	0		\\\hline
59	&	0	&	5	&	8	&	1		\\\hline
67	&	0	&	5	&	9	&	2		\\\hline
71	&	0	&	6	&	8	&	3		\\\hline
79	&	0	&	3	&	13	&	3		\\\hline
83	&	0	&	7	&	12	&	1		\\\hline
103	&	0	&	7	&	14	&	4		\\\hline
107	&	0	&	9	&	12	&	5		\\\hline
\multicolumn{5}{|c|}{$\vdots$}											\\\hline
1423	&	0	&	0	&	15	&	340		\\\hline
1427	&	0	&	0	&	17	&	339		\\\hline
1439	&	0	&	0	&	17	&	342		\\\hline
1447	&	0	&	0	&	22	&	339 	\\\hline
1451	&	0	&	0	&	20	&	342		\\\hline
1459	&	0	&	0	&	16	&	348		\\\hline
1471	&	0	&	0	&	15	&	352		\\\hline
1483	&	0	&	0	&	8	&	362		\\\hline
1487	&	0	&	0	&	12	&	359		\\\hline
\end{tabular}
\quad
\begin{tabular}{|c||*{4}{c|}}\hline
\diagbox[height=6mm]{$p$}{$\delta_f$}	& 2 &	3	&	4	&	5			\\\hline
3931	&	0	&	0	&	2	&	980	\\\hline
3943	&	0	&	0	&	0	&	985	\\\hline
3947	&	0	&	0	&	2	&	984	\\\hline
3967	&	0	&	0	&	0	&	991	\\\hline
4003	&	0	&	0	&	2	&	998		\\\hline
4007	&	0	&	0	&	1	&	1000	\\\hline
4019	&	0	&	0	&	1	&	1003	\\\hline
4027	&	0	&	0	&	1	&	1005	\\\hline
4051	&	0	&	0	&	1	&	1011	\\\hline
4079	&	0	&	0	&	0	&	1019	\\\hline
4091	&	0	&	0	&	1	&	1021	\\\hline
4099	&	0	&	0	&	1	&	1023	\\\hline
4111	&	0	&	0	&	0	&	1027	\\\hline
4127	&	0	&	0	&	0	&	1031	\\\hline
\multicolumn{5}{|c|}{$\vdots$}											\\\hline
6899	& 0	&0&	1&	1723\\\hline
6907	&	0	&	0	&	1	&	1725	\\\hline
6911	&	0	&	0	&	0	&	1727	\\\hline
6947	&	0	&	0	&	1	&	1735	\\\hline
6959	&	0	&	0	&	0	&	1739	\\\hline
6967	&	0	&	0	&	0	&	1741	\\\hline
6971	&	0	&	0	&	1	&	1741	\\\hline
6983	&	0	&	0	&	0	&	1745	\\\hline
6991	&	0	&	0	&	0	&	1747	\\\hline
\end{tabular}
\end{table}
\begin{table}
\caption{\textbf{$s=4$ in Theorem \ref{th:DU_sEven_d2}:} The number of PPs $f$ that have each possible $\delta_f$, where $f$ is of the form $f(x)=x^4(\eta(x)+b)\in\ffx{p}$, $2\le b\le (p-1)/2$. }\label{tb:s4}
\begin{tabular}{|c||*{12}{c|}}\hline
\diagbox[height=6mm]{$p$}{$\delta_f$}	& 2 &	3	&	4	&	5	&	6	&	7	&	8	&	9	&	10	&	11	&	12	&	13		\\\hline
7	&	1	&	0	&	0	&	0	&	0	&	0	&	0	&	0	&	0	&	0	&	0	&	0	\\\hline
11	&	0	&	2	&	0	&	0	&	0	&	0	&	0	&	0	&	0	&	0	&	0	&	0	\\\hline
19	&	1	&	2	&	1	&	0	&	0	&	0	&	0	&	0	&	0	&	0	&	0	&	0	\\\hline
23	&	0	&	4	&	1	&	0	&	0	&	0	&	0	&	0	&	0	&	0	&	0	&	0	\\\hline
31	&	0	&	6	&	1	&	0	&	0	&	0	&	0	&	0	&	0	&	0	&	0	&	0	\\\hline
43	&	0	&	6	&	3	&	1	&	0	&	0	&	0	&	0	&	0	&	0	&	0	&	0	\\\hline
47	&	0	&	6	&	5	&	0	&	0	&	0	&	0	&	0	&	0	&	0	&	0	&	0	\\\hline
59	&	0	&	3	&	7	&	3	&	1	&	0	&	0	&	0	&	0	&	0	&	0	&	0	\\\hline
67	&	0	&	6	&	7	&	3	&	0	&	0	&	0	&	0	&	0	&	0	&	0	&	0	\\\hline
71	&	0	&	5	&	8	&	3	&	0	&	0	&	1	&	0	&	0	&	0	&	0	&	0	\\\hline
79	&	0	&	4	&	8	&	4	&	2	&	0	&	1	&	0	&	0	&	0	&	0	&	0	\\\hline
83	&	0	&	0	&	14	&	4	&	1	&	1	&	0	&	0	&	0	&	0	&	0	&	0	\\\hline
103	&	0	&	6	&	13	&	6	&	0	&	0	&	0	&	0	&	0	&	0	&	0	&	0	\\\hline
107	&	0	&	0	&	15	&	8	&	3	&	0	&	0	&	0	&	0	&	0	&	0	&	0	\\\hline
127	&	0	&	1	&	17	&	12	&	1	&	0	&	0	&	0	&	0	&	0	&	0	&	0	\\\hline
\multicolumn{13}{|c|}{$\vdots$}																									\\\hline
3319	&	0	&	0	&	0	&	123	&	525	&	148	&	24	&	9	&	0	&	0	&	0	&	0	\\\hline
3323	&	0	&	0	&	2	&	172	&	510	&	127	&	19	&	0	&	0	&	0	&	0	&	0	\\\hline
3331	&	0	&	0	&	0	&	113	&	523	&	170	&	24	&	2	&	0	&	0	&	0	&	0	\\\hline
\multicolumn{13}{|c|}{$\vdots$}																									\\\hline
3631	&	0	&	0	&	0	&	143	&	543	&	189	&	31	&	1	&	0	&	0	&	0	&	0	\\\hline
3643	&	0	&	0	&	0	&	101	&	561	&	215	&	29	&	4	&	0	&	0	&	0	&	0	\\\hline
3659	&	0	&	0	&	0	&	89	&	575	&	214	&	32	&	3	&	1	&	0	&	0	&	0	\\\hline
3671	&	0	&	0	&	0	&	87	&	566	&	224	&	34	&	5	&	0	&	0	&	0	&	1	\\\hline
3691	&	0	&	0	&	0	&	119	&	591	&	188	&	24	&	0	&	0	&	0	&	0	&	0	\\\hline
3719	&	0	&	0	&	0	&	81	&	587	&	229	&	30	&	2	&	0	&	0	&	0	&	0	\\\hline
3727	&	0	&	0	&	0	&	134	&	578	&	190	&	28	&	1	&	0	&	0	&	0	&	0	\\\hline
\multicolumn{13}{|c|}{$\vdots$}																									\\\hline
5003	&	0	&	0	&	0	&	79	&	750	&	364	&	51	&	6	&	0	&	0	&	0	&	0	\\\hline
5011	&	0	&	0	&	0	&	91	&	805	&	301	&	47	&	6	&	2	&	0	&	0	&	0	\\\hline
5023	&	0	&	0	&	0	&	56	&	764	&	363	&	65	&	7	&	0	&	0	&	0	&	0	\\\hline
5039	&	0	&	0	&	0	&	81	&	779	&	339	&	59	&	1	&	0	&	0	&	0	&	0	\\\hline
5051	&	0	&	0	&	0	&	67	&	758	&	374	&	59	&	3	&	1	&	0	&	0	&	0	\\\hline
\multicolumn{13}{|c|}{$\vdots$}																									\\\hline
9839	&	0	&	0	&	0	&	2	&	1079	&	1144	&	219	&	14	&	1	&	0	&	0	&	0	\\\hline
9851	&	0	&	0	&	0	&	2	&	1153	&	1107	&	184	&	13	&	2	&	1	&	0	&	0	\\\hline
9859	&	0	&	0	&	0	&	16	&	1306	&	988	&	138	&	15	&	0	&	1	&	0	&	0	\\\hline
9871	&	0	&	0	&	0	&	15	&	1158	&	1096	&	178	&	19	&	1	&	0	&	0	&	0	\\\hline
9883	&	0	&	0	&	0	&	6	&	1090	&	1138	&	220	&	16	&	0	&	0	&	0	&	0	\\\hline
9887	&	0	&	0	&	0	&	9	&	1189	&	1076	&	168	&	29	&	0	&	0	&	0	&	0	\\\hline
9907	&	0	&	0	&	0	&	1	&	1002	&	1213	&	221	&	34	&	5	&	0	&	0	&	0	\\\hline
9923	&	0	&	0	&	0	&	15	&	1185	&	1076	&	187	&	17	&	0	&	0	&	0	&	0	\\\hline
9931	&	0	&	0	&	0	&	4	&	1003	&	1236	&	214	&	23	&	2	&	0	&	0	&	0	\\\hline
9967	&	0	&	0	&	0	&	20	&	1328	&	1000	&	130	&	11	&	2	&	0	&	0	&	0	\\\hline
\end{tabular}
\end{table}
\begin{table}
\caption{\textbf{$s=6$ in Theorem \ref{th:DU_sEven_d2}:} The number of PPs $f$ that have each possible $\delta_f$, where $f$ is of the form $f(x)=x^6(\eta(x)+b)\in\ffx{p}$, $2\le b\le (p-1)/2$. }\label{tb:s6}
\begin{tabular}{|c||*{13}{c|}}\hline
\diagbox[height=6mm]{$p$}{$\delta_f$}	& 2 &	3	&	4	&	5	&	6	&	7	&	8	&	9	&	10	&	11	&	12	&	13	&	$\ge 14$	\\\hline
11	&	0	&	2	&	0	&	0	&	0	&	0	&	0	&	0	&	0	&	0	&	0	&	0	&	0	\\\hline
23	&	1	&	3	&	1	&	0	&	0	&	0	&	0	&	0	&	0	&	0	&	0	&	0	&	0	\\\hline
47	&	0	&	6	&	5	&	0	&	0	&	0	&	0	&	0	&	0	&	0	&	0	&	0	&	0	\\\hline
59	&	0	&	4	&	8	&	2	&	0	&	0	&	0	&	0	&	0	&	0	&	0	&	0	&	0	\\\hline
71	&	0	&	1	&	8	&	5	&	0	&	1	&	2	&	0	&	0	&	0	&	0	&	0	&	0	\\\hline
83	&	0	&	2	&	14	&	4	&	0	&	0	&	0	&	0	&	0	&	0	&	0	&	0	&	0	\\\hline
107	&	0	&	3	&	16	&	5	&	2	&	0	&	0	&	0	&	0	&	0	&	0	&	0	&	0	\\\hline
131	&	0	&	3	&	16	&	10	&	2	&	1	&	0	&	0	&	0	&	0	&	0	&	0	&	0	\\\hline
167	&	0	&	2	&	26	&	7	&	5	&	1	&	0	&	0	&	0	&	0	&	0	&	0	&	0	\\\hline
179	&	0	&	3	&	26	&	14	&	1	&	0	&	0	&	0	&	0	&	0	&	0	&	0	&	0	\\\hline
191	&	0	&	1	&	29	&	15	&	2	&	0	&	0	&	0	&	0	&	0	&	0	&	0	&	0	\\\hline
\multicolumn{14}{|c|}{$\vdots$}\\\hline																											
2411	&	0	&	0	&	1	&	220	&	313	&	65	&	3	&	0	&	0	&	0	&	0	&	0	&	0	\\\hline
2423	&	0	&	0	&	0	&	91	&	374	&	115	&	20	&	5	&	0	&	0	&	0	&	0	&	0	\\\hline
2447	&	0	&	0	&	0	&	148	&	352	&	95	&	13	&	3	&	0	&	0	&	0	&	0	&	0	\\\hline
2459	&	0	&	0	&	0	&	138	&	374	&	82	&	19	&	1	&	0	&	0	&	0	&	0	&	0	\\\hline
2531	&	0	&	0	&	0	&	167	&	358	&	86	&	19	&	2	&	0	&	0	&	0	&	0	&	0	\\\hline
2543	&	0	&	0	&	0	&	122	&	349	&	130	&	26	&	7	&	1	&	0	&	0	&	0	&	0	\\\hline
2579	&	0	&	0	&	1	&	166	&	375	&	89	&	13	&	0	&	0	&	0	&	0	&	0	&	0	\\\hline
\multicolumn{14}{|c|}{$\vdots$}\\\hline																											
5003	&	0	&	0	&	0	&	108	&	750	&	353	&	37	&	2	&	0	&	0	&	0	&	0	&	0	\\\hline
5039	&	0	&	0	&	0	&	95	&	785	&	332	&	37	&	8	&	2	&	0	&	0	&	0	&	0	\\\hline
5051	&	0	&	0	&	0	&	13	&	480	&	591	&	135	&	35	&	5	&	3	&	0	&	0	&	0	\\\hline
5087	&	0	&	0	&	0	&	61	&	817	&	331	&	61	&	1	&	0	&	0	&	0	&	0	&	0	\\\hline
5099	&	0	&	0	&	0	&	129	&	777	&	318	&	46	&	4	&	0	&	0	&	0	&	0	&	0	\\\hline
\multicolumn{14}{|c|}{$\vdots$}\\\hline																											
5867	&	0	&	0	&	0	&	91	&	936	&	392	&	41	&	6	&	0	&	0	&	0	&	0	&	0	\\\hline
5879	&	0	&	0	&	0	&	32	&	729	&	587	&	109	&	12	&	0	&	0	&	0	&	0	&	0	\\\hline
5903	&	0	&	0	&	0	&	21	&	574	&	636	&	179	&	50	&	10	&	2	&	2	&	1	&	0	\\\hline
5927	&	0	&	0	&	0	&	38	&	854	&	508	&	72	&	8	&	1	&	0	&	0	&	0	&	0	\\\hline
5939	&	0	&	0	&	0	&	41	&	827	&	526	&	73	&	16	&	1	&	0	&	0	&	0	&	0	\\\hline
\multicolumn{14}{|c|}{$\vdots$}\\\hline																											
9491	&	0	&	0	&	0	&	3	&	736	&	1273	&	301	&	54	&	4	&	0	&	1	&	0	&	0	\\\hline
9539	&	0	&	0	&	0	&	9	&	944	&	1148	&	251	&	30	&	2	&	0	&	0	&	0	&	0	\\\hline
9551	&	0	&	0	&	0	&	15	&	1173	&	1015	&	156	&	26	&	2	&	0	&	0	&	0	&	0	\\\hline
9587	&	0	&	0	&	0	&	18	&	1149	&	1008	&	199	&	21	&	1	&	0	&	0	&	0	&	0	\\\hline
9623	&	0	&	0	&	0	&	15	&	1181	&	1045	&	143	&	19	&	2	&	0	&	0	&	0	&	0	\\\hline
9719	&	0	&	0	&	0	&	2	&	761	&	1290	&	317	&	54	&	5	&	0	&	0	&	0	&	0	\\\hline
9743	&	0	&	0	&	0	&	12	&	1175	&	1069	&	161	&	13	&	4	&	1	&	0	&	0	&	0	\\\hline
9767	&	0	&	0	&	0	&	5	&	1028	&	1182	&	197	&	25	&	4	&	0	&	0	&	0	&	0	\\\hline
9791	&	0	&	0	&	0	&	21	&	1247	&	994	&	159	&	23	&	3	&	0	&	0	&	0	&	0	\\\hline
9803	&	0	&	0	&	0	&	6	&	1111	&	1120	&	184	&	26	&	3	&	0	&	0	&	0	&	0	\\\hline
9839	&	0	&	0	&	0	&	1	&	879	&	1287	&	258	&	29	&	5	&	0	&	0	&	0	&	0	\\\hline
9851	&	0	&	0	&	0	&	5	&	1051	&	1169	&	214	&	21	&	2	&	0	&	0	&	0	&	0	\\\hline
9887	&	0	&	0	&	0	&	25	&	1186	&	1049	&	189	&	17	&	5	&	0	&	0	&	0	&	0	\\\hline
9923	&	0	&	0	&	0	&	12	&	1077	&	1174	&	193	&	20	&	4	&	0	&	0	&	0	&	0	\\\hline
\end{tabular}
\end{table}

\end{document}